\documentclass[12pt,a4paper]{amsart}
\usepackage{amssymb}
\usepackage{hyperref}

\calclayout

\newcommand{\+}{\nobreakdash-}

\newcommand{\rarrow}{\longrightarrow}

\newcommand{\ot}{\otimes}

\newcommand{\bu}{{\text{\smaller\smaller$\scriptstyle\bullet$}}}

\newcommand{\lrarrow}{\mskip.5\thinmuskip\relbar\joinrel\relbar\joinrel
 \rightarrow\mskip.5\thinmuskip\relax}

\DeclareMathOperator{\Hom}{Hom}
\DeclareMathOperator{\Ext}{Ext}

\newcommand{\Modl}{{\operatorname{\mathsf{--Mod}}}}

\newcommand{\sC}{\mathsf C}
\newcommand{\sF}{\mathsf F}
\newcommand{\sS}{\mathsf S}

\newcommand{\proj}{\mathsf{proj}}
\renewcommand{\flat}{\mathsf{flat}}
\renewcommand{\cot}{\mathsf{cot}}

\newcommand{\dcot}{{\operatorname{\mathsf{-cot}}}}
\newcommand{\dflpr}{{\operatorname{\mathsf{-flpr}}}}

\newcommand{\rop}{\mathrm{op}}

\newcommand{\Fil}{\mathsf{Fil}}

\newcommand{\boZ}{\mathbb Z}
\newcommand{\boQ}{\mathbb Q}

\newcommand{\Section}[1]{\bigskip\section{#1}\medskip}
\theoremstyle{plain}
\newtheorem{thm}{Theorem}[section]
\newtheorem{prop}[thm]{Proposition}
\newtheorem{lem}[thm]{Lemma}
\newtheorem{cor}[thm]{Corollary}
\theoremstyle{definition}
\newtheorem{rem}[thm]{Remark}
\newtheorem{rems}[thm]{Remarks}
\newtheorem{quest}[thm]{Question}
\newtheorem{ex}[thm]{Example}

\begin{document}

\title{A relative version of Bass' theorem \\
about finite-dimensional algebras}

\author{Leonid Positselski}

\address{Institute of Mathematics, Czech Academy of Sciences \\
\v Zitn\'a~25, 115~67 Prague~1 \\ Czech Republic} 

\email{positselski@math.cas.cz}

\begin{abstract}
 As a special case of Bass' theory of perfect rings, one obtains
the assertion that, over a finite-dimensional associative algebra over
a field, all flat modules are projective.
 In this paper we prove the following relative version of this result.
 Let $R\rarrow A$ be a homomorphism of associative rings such that $A$
is a finitely generated projective right $R$\+module.
 Then every flat left $A$\+module is a direct summand of an $A$\+module
filtered by $A$\+modules $A\ot_RF$ induced from flat left
$R$\+modules~$F$.
 In other words, a left $A$\+module is cotorsion if and only if its
underlying left $R$\+module is cotorsion.
 The proof is based on the cotorsion periodicity theorem.
\end{abstract}

\maketitle

\tableofcontents

\section*{Introduction}
\medskip

 All finite-dimensional associative algebras $A$ over a field~$k$ are
left and right Artinian rings.
 Any associative ring that is either left or right Artinian is left
perfect.
 According to~\cite[Theorem~P]{Bas}, it follows that all flat
$A$\+modules are projective.

 This proof of projectivity of flat modules is based on the structure
theory of finite-dimensional algebras over a field.
 Any such algebra $A$ has a radical $J$, which is a nilpotent ideal
in $A$, and the quotient algebra $A/J$ is classically semisimple.

 According to the original Bass' definition in~\cite{Bas},
an associative ring $A$ is said to be \emph{left perfect} if all left
$A$\+modules have projective covers.
 By~\cite[Theorem~P(3)]{Bas}, the ring $A$ is left perfect if and only
if all flat left $A$\+modules are projective.
 The condition of~\cite[Theorem~P(1)]{Bas} provides a characterization
of left perfect rings in terms of their structural properties that is
used in the argument above.

 Now let $R$ be a commutative ring and $A$ be an associative
$R$\+algebra that is finitely generated and projective as
an $R$\+module.
 What can one say about flat $A$\+modules in this relative context?
 Notice that the structure theory of finite-dimensional algebras over
a field is not available for finite algebras over commutative rings.

 Our approach to the question in the previous paragraph is based on
the notion of a \emph{cotorsion pair}, which was introduced by
Salce~\cite{Sal} in the context of the theory of infinitely generated
abelian groups.
 An introductory discussion of cotorsion pairs in abelian categories
and module categories can be found in the introduction to
the paper~\cite{Pctrl}.
 The book~\cite{GT} contains a wealth of results and examples.

 Let $A$ be an associative ring.
 A \emph{cotorsion pair} $(\sF,\sC)$ in the category of left
$A$\+modules $A\Modl$ is a pair of classes of modules $\sF$,
$\sC\subset A\Modl$ such that $\Ext^1_A(F,C)=0$ for all $F\in\sF$,
\ $C\in\sC$, and both the classes $\sF$ and $\sC$ are maximal for
this property with respect to each other.

 The thematic example of a cotorsion pair is the \emph{flat cotorsion
pair} (also known as the \emph{Enochs cotorsion pair}~\cite{En2},
\cite[Definition~5.18(i)]{GT}).
 In this example, the class $\sF=A\Modl_\flat$ is the class of all
flat left $A$\+modules.
 The class $\sC=A\Modl^\cot$ is the class of all \emph{cotorsion}
left $A$\+modules.
 Here a left $A$\+module $C$ is said to be \emph{cotorsion} if
$\Ext^1_A(F,C)=0$ for all flat left $A$\+modules $F$, or equivalently,
$\Ext^n_A(F,C)=0$ for all flat left $A$\+modules $F$ and all $n\ge1$.

 The main result about cotorsion pairs is the \emph{Eklof--Trlifaj
theorem}~\cite[Theorems~2 and~10]{ET}, \cite[Theorem~6.11 and
Corollary~6.14]{GT}.
 The Eklof--Trlifaj theorem concerns the notion of a \emph{complete}
cotorsion pair, and its simplest formulation claims that \emph{every
cotorsion pair generated by a set of modules is complete}.
 Proving completeness of the flat cotorsion pair in the category of
modules over an arbitrary ring was a major open problem, which
became known under the name of the \emph{Flat Cover Conjecture}.
 Posed by Enochs in~\cite{En} in~1981 (see also~\cite[Section~1.3]{Xu}),
it was proved by Bican--El~Bashir--Enochs in~\cite{BBE} in~2001.
 One of the two proofs of the Flat Cover Conjecture presented
in~\cite{BBE} was based on the Eklof--Trlifaj theorem.
 Subsequently it was understood~\cite[Corollary~6.6]{Hov},
\cite[Theorem~4.5]{Ros} that a proof of the Eklof--Trlifaj theorem
can be also obtained as an application of Quillen's small object
argument~\cite[Lemma~II.3.3]{Quil}.

 A trivial example of a cotorsion pair in $A\Modl$ is
the \emph{projective cotorsion pair}.
 It is formed by the two classes of projective $A$\+modules
$\sF=A\Modl_\proj$ and arbitrary $A$\+modules $\sC=R\Modl$.
 The assertion that the projective cotorsion pair is complete means
precisely that there are enough projective objects in $A\Modl$.

 Another cotorsion pair relevant for the purposes of the present paper
is the \emph{flaprojective cotorsion pair}.
 It depends on a homomorphism of associative rings $R\rarrow A$.
 The flaprojective cotorsion pair in $A\Modl$ is a relative mixture of
the flat cotorsion pair in $R\Modl$ and the projective cotorsion pair
in the direction of $A$ relative to~$R$.

 Specifically, the flaprojective cotorsion pair $(\sF,\sC)$ is formed
by the following two classes of left $A$\+modules.
 The right-hand class $\sC$ consists of all left $A$\+modules that are
\emph{cotorsion as $R$\+modules}, i.~e., the $A$\+modules whose
underlying $R$\+module is cotorsion.
 The left-hand class $\sF$ consists of all the left $A$\+modules $F$
such that $\Ext^1_A(F,C)=0$ for all left $A$\+modules $C$ that are
cotorsion over~$R$.
 The $A$\+modules $F$ satisfying the latter condition are called
\emph{flaprojective relative to~$R$} (or \emph{$A/R$\+flaprojective})
\cite[Section~1.6]{Pdomc}.

 The definition of a flaprojective module in the previous paragraph is
very indirect and roundabout.
 First one defines cotorsion $R$\+modules as the $R$\+modules right
$\Ext^1_R$\+orthogonal to flat $R$\+modules, and then one defines
flaprojective $A$\+modules as the $A$\+modules left
$\Ext^1_A$\+orthogonal to $R$\+cotorsion $A$\+modules.
 A much more direct and explicit description is provided by
the Eklof--Trlifaj theorem, but it uses the concept of
an \emph{ordinal-indexed filtration} on modules.
 A left $A$\+module $F$ is $A/R$\+flaprojective if and only if $F$
is a direct summand of an $A$\+module admitting an ordinal-indexed
filtration whose successive quotients are the $A$\+modules
$A\ot_RG$ induced from flat left $R$\+modules~$G$.

 Now that we have explained the meaning of the concepts involved,
we can formulate the main result of this paper.

\begin{thm} \label{main-theorem}
 Let $R\rarrow A$ be a homomorphism of associative rings making $A$
a finitely generated projective right $R$\+module.
 Then the flat cotorsion pair in $A\Modl$ consides with
the $A/R$\+flaprojective cotorsion pair.
 More specifically: \par
\textup{(a)} A left $A$\+module $C$ is cotorsion if and only if $C$ is
cotorsion as a left $R$\+module. \par
\textup{(b)} A left $A$\+module is flat if and only if it is
$A/R$\+flaprojective.
 In other words, a left $A$\+module $F$ is flat if and only if $F$ is
a direct summand of an $A$\+module filtered by $A$\+modules $A\ot_RG$
induced from flat left $R$\+modules~$G$.
\end{thm}

 It would be interesting to know how far one can generalize
Theorem~\ref{main-theorem}.
 Quite generally, for any homomorphism of associative rings
$R\rarrow A$, property~(a) in the theorem holds if and only if
property~(b) holds.
 If this is the case, one can say that the ring homomorphism
$R\rarrow A$ is \emph{relatively left perfect}.

 In the absolute setting, a ring $A$ is called \emph{left perfect}
if all flat left $A$\+modules are projective~\cite[Theorem~P]{Bas}.
 For example, if $A$ is an associative algebra over a field~$k$,
then the ring $A$ is left perfect in the sense of~\cite{Bas} if and
only if the ring homomorphism $k\rarrow A$ is relatively left
perfect in the sense of the previous paragraph.

 There are many examples of left perfect rings and left perfect
$k$\+algebras beyond the finite-dimensional $k$\+algebras.
 A class of examples of relatively left perfect ring homomorphisms
is provided Theorem~\ref{main-theorem}.
 Another such class, consisting of all the left flat ring epimorphisms
$R\rarrow A$, is discussed in
Example~\ref{left-flat-ring-epimorphisms-example} at the end of
this paper.
 What are other nontrivial classes of examples of relatively left
perfect ring homomorphisms?
 Can one characterize relatively left perfect ring homomorphisms by any
relative analogues of the properties listed in~\cite[Theorem~P]{Bas}\,?

\subsection*{Acknowledgement}
 I~am grateful to Michal Hrbek and Sergio Pavon for helpful
conversations.
 A special gratitute goes to Silvana Bazzoni for communicating
Example~\ref{left-flat-ring-epimorphisms-example} to me and giving
a kind permission to reproduce it in this paper.
 The author is supported by the GA\v CR project 23-05148S
and the Institute of Mathematics, Czech Academy of Sciences
(research plan RVO:~67985840).

\Section{Preliminaries on Cotorsion Pairs}

 Let $A$ be an associative ring.
 We denote by $A\Modl$ the abelian category of left $A$\+modules.

 Let $\sF$, $\sC\subset A\Modl$ be two classes of $A$\+modules.
 Then one denotes by $\sF^{\perp_1}\subset A\Modl$ the class of all
left $A$\+modules $X$ such that $\Ext^1_A(F,X)=0$ for all $F\in\sF$.
 Dually, the notation ${}^{\perp_1}\sC\subset A\Modl$ stands for
the class of all left $A$\+modules $Y$ such that $\Ext^1_A(Y,C)=0$
for all $C\in\sC$.

 A pair of classes $(\sF,\sC)$ in $A\Modl$ is said to be
a \emph{cotorsion pair} if $\sC=\sF^{\perp_1}$ and
$\sF={}^{\perp_1}\sC$.
 For any class of modules $\sS\subset A\Modl$, the pair of classes
$\sC=\sS^{\perp_1}$ and $\sF={}^{\perp_1}\sC$ is a cotorsion pair
in $A\Modl$.
 The latter cotorsion pair $(\sF,\sC)$ is said to be \emph{generated
by} the class of modules $\sS\subset A\Modl$.

 A cotorsion pair $(\sF,\sC)$ in $A\Modl$ is said to be \emph{complete}
if for every left $A$\+module $M$ there exist short exact sequences
of $A$\+modules
\begin{alignat}{4}
 0 &\lrarrow C' &&\lrarrow F &&\lrarrow M &&\lrarrow 0,
 \label{special-precover-sequence} \\
 0 &\lrarrow M &&\lrarrow C &&\lrarrow F' &&\lrarrow 0
 \label{special-preenvelope-sequence}
\end{alignat}
with $A$\+modules $F$, $F'\in\sF$ and $C$, $C'\in\sC$.
 The short exact sequence~\eqref{special-precover-sequence} is called
a \emph{special precover sequence}.
 The short exact sequence~\eqref{special-preenvelope-sequence} is
called a \emph{special preenvelope sequence}.
 Taken together, the short exact
sequences~(\ref{special-precover-sequence}\+-%
\ref{special-preenvelope-sequence}) are called
the \emph{approximation sequences}.

 An introductory discussion of cotorsion pairs can be found in
the introduction to~\cite{Pctrl}.
 Many examples of approximation sequences (for abelian groups,
i.~e., modules over the ring of integers $A=\boZ$) are discussed
in~\cite[Section~12]{Pcta}.
 The main reference source on cotorsion pairs in module categories
is the book~\cite{GT}.

 Let $F$ be an $A$\+module and $\alpha$~be an ordinal.
 An \emph{$\alpha$\+indexed filtration} on $F$ is a family of
submodules $F_i\subset F$, indexed by the ordinals $0\le i\le\alpha$
and satisfying the following conditions:
\begin{itemize}
\item $F_0=0$ and $F_\alpha=F$;
\item $F_i\subset F_j$ for all ordinals $0\le i\le j\le\alpha$;
\item $F_j=\bigcup_{i<j}F_i$ for all limit ordinals $0<j\le\alpha$.
\end{itemize}
 Given an $A$\+module $F$ with an ordinal-indexed filtration
$(F_i\subset F)_{0\le i\le\alpha}$, one considers the successive
quotient $A$\+modules $S_i=F_{i+1}/F_i$, where $0\le i<\alpha$.
 The $A$\+module $F$ is said to be \emph{filtered by}
the $A$\+modules~$S_i$.
 In an alternative terminology, the $A$\+module $F$ is said to be
a \emph{transfinitely iterated extension} (\emph{in the sense of
the direct limit}) of the $A$\+modules $S_i=F_{i+1}/F_i$.

 For any class of $A$\+modules $\sS\subset A\Modl$, one denotes
by $\Fil(\sS)$ the class of all $A$\+modules filtered by (modules
isomorphic to) modules from~$\sS$.
 A class of $A$\+modules $\sF\subset A\Modl$ is said to be
\emph{deconstructible} if there exists a \emph{set} (rather than
a proper class) of modules $\sS\subset A\Modl$ such that
$\sF=\Fil(\sS)$.

 The following result is known as the \emph{Eklof lemma}.

\begin{lem} \label{eklof-lemma}
 For any class of $A$\+modules\/ $\sS\subset A\Modl$, one has\/
$\sS^{\perp_1}=\Fil(\sS)^{\perp_1}$.
 In other words, for any class of $A$\+modules\/ $\sC\subset A\Modl$,
the class of $A$\+modules\/ ${}^{\perp_1}\sC\subset A\Modl$ is
closed under transfinitely iterated extensions, that is,
${}^{\perp_1}\sC=\Fil({}^{\perp_1}\sC)$.
\end{lem}

\begin{proof}
 This is~\cite[Lemma~1]{ET} or~\cite[Lemma~6.2]{GT}.
\end{proof}

 Given a class of $A$\+modules $\sF\subset A\Modl$, we denote by
$\sF^\oplus\subset A\Modl$ the class of all direct summands of
modules from~$\sF$.

\begin{thm}[Eklof--Trlifaj~\cite{ET}] \label{eklof-trlifaj-theorem}
 Let $A$ be an associative ring and\/ $\sS\subset A\Modl$ be
a \emph{set} of left $A$\+modules.
 In this context: \par
\textup{(a)} The cotorsion pair $(\sF,\sC)$ generated by\/ $\sS$ in
$A\Modl$ is complete. \par
\textup{(b)} Assume that the free left $A$\+module $A$ with one
generator belongs to\/~$\sS$, that is ${}_AA\in\sS$.
 Then one has\/ $\sF=\Fil(\sS)^\oplus$.
\end{thm}

\begin{proof}
 This is~\cite[Theorems~2 and~10]{ET} or~\cite[Theorem~6.11 and
Corollary~6.14]{GT}.
\end{proof}

\begin{cor} \label{eklof-trlifaj-cor}
 Let $A$ be an associative ring and\/ $\sF\subset A\Modl$ be a class
of left $A$\+modules.
 In this context: \par
\textup{(a)} If the class\/ $\sF$ is deconstructible, then
the cotorsion pair generated by\/ $\sF$ in $A\Modl$ is complete. \par
\textup{(b)} If the class\/ $\sF$ is deconstructible, closed under
direct summands in $A\Modl$, and contains the free $A$\+module
${}_AA$, then\/ $\sF$ is the left-hand part of a complete cotorsion
pair $(\sF,\sC)$ in $A\Modl$.
 In other words, the pair of classes $(\sF,\sF^{\perp_1})$ is
a complete cotorsion pair in $A\Modl$.
\end{cor}

\begin{proof}
 Follows from Lemma~\ref{eklof-lemma} and
Theorem~\ref{eklof-trlifaj-theorem}.
\end{proof}

\Section{Cotorsion Modules}  \label{cotorsion-modules-secn}

 Let $R$ be an associative ring.
 A left $R$\+module $C$ is said to be \emph{cotorsion}
(or \emph{Enochs cotorsion}~\cite{En2}) if $\Ext^1_R(F,C)=0$ for all
flat left $R$\+modules~$F$.

 Using the facts that all projective left $R$\+modules are flat and
the class of flat modules is closed under kernels of epimorphisms in
$R\Modl$, one can easily show that $\Ext_R^n(F,C)=0$ for all flat
left $R$\+modules $F$, all cotorsion left $R$\+modules $C$, and all
integers $n\ge1$ (cf.~\cite[Lemma~5.24]{GT} and~\cite[Lemma~7.1]{PS6}).
 It follows that the class of cotorsion $R$\+modules is closed under
cokernels of monomorphisms in $R\Modl$; it is also obviously closed
under extensions and infinite products.
 We denote the class of flat left $R$\+modules by $R\Modl_\flat
\subset R\Modl$ and the class of cotorsion left $R$\+modules by
$R\Modl^\cot\subset R\Modl$.

 The most nontrivial known property of cotorsion modules is
the following \emph{cotorsion periodicity theorem} of Bazzoni,
Cort\'es-Izurdiaga, and Estrada~\cite{BCE}.

\begin{thm} \label{cotorsion-periodicity}
 Let $A$ be an associative ring and $C^\bu$ be an (unbounded) acyclic
complex of $A$\+modules such that all the $A$\+modules $C^n$,
\,$n\in\boZ$ are cotorsion.
 Then the $A$\+modules of cocycles of the acyclic complex $C^\bu$ are
cotorsion $A$\+modules, too.
\end{thm}

\begin{proof}
 This is~\cite[Theorem~5.1(2)]{BCE}.
 See also~\cite[Theorem~1.2(2) or Proposition~4.8(2)]{BCE} for
essentially equivalent formulations of this theorem.
 A further discussion in category-theoretic context can be found
in~\cite[Corollary~9.2 and Proposition~9.4]{PS6}.
\end{proof}

\begin{cor} \label{cotorsion-resolution-cor}
 Let $A$ be an associative ring and $M$ be an $A$\+module.
 Assume that $M$ admits a (left) resolution by cotorsion $A$\+modules,
i.~e., there exists an exact sequence of $A$\+modules
$$
 \dotsb\lrarrow C_2\lrarrow C_1\lrarrow C_0\overset p\lrarrow M\lrarrow0
$$
with cotorsion $A$\+modules $C_n$, \,$n\ge0$.
 Then $M$ is a cotorsion $A$\+module.
\end{cor}

\begin{proof}
 This is~\cite[Corollary~1.6.4]{Pcosh}.
 The point is that any $A$\+module can be embedded into a cotorsion
$A$\+module (in fact, all injective $A$\+modules are cotorsion by
definition).
 Therefore, any $A$\+module $M$ has a coresolution (right resolution)
by cotorsion $A$\+modules, i.~e., there is an exact sequence of
$A$\+modules
$$
 0\lrarrow M\overset i\lrarrow C^0\lrarrow C^1\lrarrow C^2
 \lrarrow\dotsb
$$
with cotorsion $A$\+modules $C^n$, \,$n\ge0$.
 Now the combined complex (the resulting two-sided resolution of~$M$)
$$
 \dotsb\lrarrow C_2\lrarrow C_1\lrarrow C_0\overset{ip}\lrarrow
 C^0\lrarrow C^1\lrarrow C^2\lrarrow\dotsb
$$
is an acyclic complex of cotorsion $A$\+modules, and $M$ is one of its
$A$\+modules of cocycles.
 So Theorem~\ref{cotorsion-periodicity} is applicable.
\end{proof}

 The following three lemmas are more elementary.
 We use the notation $\Hom_R(L,M)$ for the abelian group of morphisms
$L\rarrow M$ in the category of left $R$\+modules $R\Modl$.
 The group of homomorphisms of right $R$\+modules $N\rarrow Q$ will
be denoted by $\Hom_{R^\rop}(N,Q)$.

\begin{lem} \label{cotorsion-restriction-of-scalars}
 Let $R\rarrow A$ be a homomorphism of associative rings and $D$ be
a cotorsion $A$\+module.
 Then the underlying $R$\+module of $D$ is a cotorsion $R$\+module.
\end{lem}

\begin{proof}
 This is~\cite[Lemma~1.3.4(a)]{Pcosh} or~\cite[Lemma~1.6(a)]{Pdomc}.
 One can use the $\Ext^1$\+ad\-junc\-tion of change of rings as
in~\cite[Proposition~VI.4.1.3]{CE} (see
Lemma~\ref{ext-from-induced-from-flat} below).
 Alternatively, the following elementary argument works.
 A left $R$\+module $C$ is cotorsion if and only if the functor
$\Hom_R({-},C)$ takes every short exact sequence of flat left
$R$\+modules to a short exact sequence of abelian groups.
 Given a short exact sequence of flat left $R$\+modules $0\rarrow H
\rarrow G\rarrow F\rarrow0$, we have a short exact sequence of flat
left $A$\+modules $0\rarrow A\ot_RH\rarrow A\ot_RG\rarrow A\ot_RF
\rarrow0$.
 By assumption, the $A$\+module $D$ is cotorsion, so the functor
$\Hom_A({-},D)$ takes the latter short exact sequence of $A$\+modules
to a short exact sequence of abelian groups.
 It remains to recall the adjunction isomorphism of abelian groups
$\Hom_A(A\ot_RM,\>D)\simeq\Hom_R(M,D)$, which holds for any left
$R$\+module~$M$.
 The assertion of the lemma is also a special case of the next
Lemma~\ref{Hom-module-cotorsion}.
\end{proof}

\begin{lem} \label{Hom-module-cotorsion}
 Let $R$ and $S$ be two associative rings, $K$ be an $R$\+$S$\+bimodule
that is flat as a left $R$\+module, and $C$ be a cotorsion left
$R$\+module.
 Then the left $S$\+module\/ $\Hom_R(K,C)$ is cotorsion.
\end{lem}

\begin{proof}
 This is~\cite[Lemma~1.3.3(a)]{Pcosh}.
 The point is that the tensor product $K\ot_SG$ is a flat left
$R$\+module for any flat left $S$\+module~$G$.
 Based on this observation, one can use the $\Ext^1$\+ad\-junc\-tion of
the tensor product and Hom functors as in~\cite[formula~(4) in
Section~XVI.4]{CE}, \cite[formula~(1.1) in Section~1.2]{Pcosh},
or~\cite[Lemma~1.1]{Pdomc}.
 Alternatively, an elementary argument can be found
in~\cite[proof of Lemma~1.5(a)]{Pdomc}.
\end{proof}

\begin{lem} \label{bimodule-dualization-lemma}
 Let $R\rarrow A$ be a homomorphism of associative rings making $A$
a finitely generated projective right $R$\+module.
 Consider the $R$\+$A$\+bimodule
$$
 K=\Hom_{R^\rop}(A,R),
$$
with the left $R$\+module structure on $K$ coming from the left
$R$\+module structure on $R$ and the right $A$\+module structure
on $K$ induced by the left $A$\+module structure on~$A$.
 Then, for every left $R$\+module $M$, there is a natural isomorphism
of left $A$\+modules
$$
 A\ot_RM\simeq\Hom_R(K,M),
$$
where the left $A$\+module structure on $A\ot_RM$ comes from the left
$A$\+module structure on $A$ and the left $A$\+module structure on\/
$\Hom_R(K,M)$ comes from the right $A$\+module structure on~$K$.
\end{lem}

\begin{proof}
 More generally, for any two associative rings $R$ and $S$, let
$B$ be an $S$\+$R$\+bimodule that is finitely generated and projective
as a right $R$\+module, and let $M$ be a left $R$\+module.
 Then the $R$\+$S$\+bimodule $L=\Hom_{R^\rop}(B,R)$ is finitely
generated and projective as a left $R$\+module, there is a natural
isomorphism of $S$\+$R$\+bimodules $B\simeq\Hom_R(L,R)$, and there is
a natural isomorphism of left $S$\+modules $B\ot_RM\simeq\Hom_R(L,M)$.
 Now one can put $B=S=A$ and $L=K$.
\end{proof}

\begin{prop} \label{induction-preserves-cotorsion}
 Let $R\rarrow A$ be a homomorphism of associative rings making $A$
a finitely generated projective right $R$\+module.
 Let $C$ be a cotorsion left $R$\+module.
 Then the left $A$\+module $A\ot_RC$ is cotorsion.
\end{prop}

\begin{proof}
 By Lemma~\ref{bimodule-dualization-lemma}, we have a left $A$\+module
isomorphism $A\ot_RC\simeq\Hom_R(K,C)$, where $K=\Hom_{R^\rop}(A,R)$.
 The $R$\+$A$\+bimodule $K$ is finitely generated and projective,
hence flat, as a left $R$\+module.
 So Lemma~\ref{Hom-module-cotorsion} is applicable for the ring $S=A$.
\end{proof}

\begin{lem} \label{bar-resolution-lemma}
 Let $R\rarrow A$ be a homomorphism of associative rings.
 Then every left $A$\+module $M$ admits a natural relative bar
resolution of the form
\begin{equation} \label{bar-resolution}
 \dotsb\rarrow A\ot_RA\ot_RA\ot_RM\rarrow
 A\ot_RA\ot_RM\rarrow A\ot_RM\rarrow M\rarrow0.
\end{equation}
\end{lem}

\begin{proof}
 The differential in the relative bar resolution is given by
the standard formula
\begin{align*}
 \partial(a_0\ot a_1\ot\dotsb\ot a_n\ot m) &=
 a_0a_1\ot a_2\ot a_3\ot\dotsb\ot a_n\ot m \\
 &- a_0\ot a_1a_2\ot a_3\ot a_4\ot\dotsb\ot a_n\ot m + \dotsb \\
 &+ (-1)^{n-1} a_0\ot a_1\ot\dotsb\ot a_{n-2}\ot a_{n-1}a_n\ot m \\
 &+ (-1)^n a_0\ot a_1\ot\dotsb\ot a_{n-1}\ot a_nm,
\end{align*}
for all $a_0$,~\dots, $a_n\in A$, \ $n\ge0$, and $m\in M$.
 It is clear from this formula that $\partial$~is a left
$A$\+module morphism (where the left $A$\+module structure on
the tensor products comes from the left action of $A$ on the leftmost
tensor factor~$A$).
 One can easily compute that~\eqref{bar-resolution} is a complex,
i.~e., $\partial^2=0$.

 The lemma claims that the complex~\eqref{bar-resolution}
is acyclic.
 In fact, the complex~\eqref{bar-resolution} is contractible
\emph{as a complex of left $R$\+modules} (but \emph{not} as a complex
of left $A$\+modules).
 The natural $R$\+linear contracting homotopy~$h$ is given by
the formula
$$
 h(a_1\ot a_2\ot\dotsb\ot a_n\ot m)=
 1\ot a_1\ot a_2\ot\dotsb\ot a_n\ot m,
$$
where $1\in A$ is the unit element.
\end{proof}

 Now we can prove the main result of this section, which is also
the first main result of this paper.

\begin{thm} \label{A-cotorsion=R-cotorsion-theorem}
 Let $R\rarrow A$ be a homomorphism of associative rings such that
$A$ is a finitely generated projective right $R$\+module.
 Then a left $A$\+module $D$ is cotorsion as an $A$\+module if and only
if it is cotorsion as an $R$\+module.
\end{thm}

\begin{proof}
 The easy ``only if'' assertion holds by
Lemma~\ref{cotorsion-restriction-of-scalars}.
 To prove the ``if'', consider the left $A$\+module $A\ot_RD$
(with the $A$\+module structure coming from the left $A$\+module
structure on~$A$).
 Since $D$ is a cotorsion left $R$\+module by assumption,
Proposition~\ref{induction-preserves-cotorsion} tells us that
$A\ot_RD$ is a cotorsion left $A$\+module.

 By Lemma~\ref{cotorsion-restriction-of-scalars}, it follows that
left $R$\+module $A\ot_RD$ (with the $R$\+module structure coming
from the left $R$\+module structure on~$A$) is cotorsion.
 Therefore, Proposition~\ref{induction-preserves-cotorsion} tells us
that the left $A$\+module $A\ot_RA\ot_RD$ (with the $A$\+module
structure coming from the left $A$\+module structure on the leftmost
tensor factor~$A$) is cotorsion.

 Proceeding in this way, we prove by induction that the left
$A$\+module $A^{\ot_R\,n}\ot_RD$ (with the $A$\+module
structure coming from the left $A$\+module structure on the leftmost
tensor factor~$A$) is cotorsion for all $n\ge1$.
 Now Lemma~\ref{bar-resolution-lemma} provides a (left) resolution
of $D$ by cotorsion $A$\+modules.
 By Corollary~\ref{cotorsion-resolution-cor}, we can conclude that $D$
is a cotorsion $A$\+module, as desired.
\end{proof}

\Section{Flaprojective Modules}

 Let $R$ be an associative ring.
 It is easy to prove that the pair of classes of flat left $R$\+modules
and cotorsion left $R$\+modules $(\sF,\sC)$~$=$ $(R\Modl_\flat$,
$R\Modl^\cot)$ is a cotorsion pair in $R\Modl$
\,\cite[Lemma~3.4.1]{Xu}, \cite[Lemma~1.3]{Pdomc}.
 Proving that this pair of classes, called the \emph{flat cotorsion
pair} or the \emph{Enochs cotorsion pair}~\cite{En2},
\cite[Definition~5.18(i)]{GT}), is a \emph{complete} cotorsion pair,
is harder.
 There are two proofs of this result in the paper~\cite{BBE}, one of
them based on the following observation of Enochs.

\begin{prop} \label{flat-modules-deconstructible}
 Let $R$ be an associative ring.
 Then the class of all flat left $R$\+modules\/ $\sF=R\Modl_\flat$ is
deconstructible in $R\Modl$.
 More precisely, let $\lambda$~be an infinite cardinal not smaller
than the cardinality\/ $|R|$ of the ring $R$, and let\/ $\sS$ be the set
of (representatives of isomorphism classes of) flat left $R$\+modules
$S$ of the cardinality\/ $|S|$ not exceeding~$\lambda$.
 Then one has\/ $\sF=\Fil(\sS)$.
\end{prop}

\begin{proof}
 This is~\cite[Lemma~1 and Proposition~2]{BBE}.
 See~\cite[Lemma~6.17]{GT} for a generalization.
\end{proof}

\begin{cor}
 For any associative ring $R$, the flat cotorsion pair
$(\sF,\sC)$~$=$ $(R\Modl_\flat$, $R\Modl^\cot)$ in $R\Modl$
is complete.
\end{cor}

\begin{proof}
 This is~\cite[Section~2]{BBE} and a particular case
of~\cite[Theorem~6.19]{GT}.
 The assertion follows from
Proposition~\ref{flat-modules-deconstructible} by virtue
of Corollary~\ref{eklof-trlifaj-cor}(b).
\end{proof}

 Now let $R\rarrow A$ be a homomorphism of associative rings.
 A left $A$\+module $F$ is said to be \emph{flaprojective relative
to~$R$} (or \emph{$A/R$\+flaprojective}) \cite[Section~1.6]{Pdomc}
if $\Ext^1_A(F,C)=0$ for all left $A$\+modules $C$ that are
\emph{cotorsion as left $R$\+modules} (i.~e., the underlying left
$R$\+module of $C$ is cotorsion).

 Since the class of cotorsion $R$\+modules is closed under cokernels
of monomorphisms (see the beginning of
Section~\ref{cotorsion-modules-secn}), so is the class of all
$A$\+modules that are cotorsion over~$R$.
 All cotorsion $A$\+modules are cotorsion over $R$ by
Lemma~\ref{cotorsion-restriction-of-scalars}; in particular, all
injective $A$\+modules are cotorsion over~$R$.
 Using these observations, one can show that $\Ext^n_A(F,C)=0$ for
all $A/R$\+fla\-pro\-jec\-tive left $R$\+modules $F$, all left
$A$\+modules $C$ that are cotorsion over $R$, and all integers $n\ge1$
(cf.~\cite[Lemma~5.24]{GT} and~\cite[Lemma~7.1]{PS6}).

 It follows that the class of $A/R$\+flaprojective $A$\+modules is
closed under kernels of epimorphisms in $A\Modl$; it is also obviously
closed under extensions and infinite direct sums.
 By Lemma~\ref{eklof-lemma}, the class of $A/R$\+flaprojective
left $A$\+modules is closed under transfinitely iterated extensions.

\begin{rems} \label{flaprojective-remarks}
 (1)~Since all cotorsion $A$\+modules are cotorsion over $R$, it
follows that all $A/R$\+flaprojective $A$\+modules are
flat~\cite[Remark~1.29]{Pdomc}.
 Obviously, all projective $A$\+modules are $A/R$\+flaprojective;
so the class of $A/R$\+flaprojective $A$\+modules is, generally
speaking, intermediate between the classes of projective and flat
$A$\+modules.

 (2)~If $A=R$ and $R\rarrow A$ is the identity ring homomorphism,
then the class of $A/R$\+flaprojective $A$\+modules coincides with
the class of flat $R$\+modules~\cite[Example~1.28(1)]{Pdomc}.

 (3)~If $R$ is a left perfect ring (in the sense
of~\cite[Theorem~P]{Bas}), then all flat left $R$\+modules are
projective, and therefore all left $R$\+modules are cotorsion.
 Consequently, the class of $A/R$\+flaprojective left $A$\+modules
coincides with the class of projective left $A$\+modules in this case.

 (4)~The converse assertion to~(3) is \emph{not} true in general.
 For example, if $R=\boZ$ is the ring of integers and $A=\boQ$ is
the field of rational numbers, then all $A$\+modules are projective,
hence they are also $A/R$\+flaprojective by~(1).
 But the ring $R$ is not perfect.
 
 We are not aware of any characterization of the ring homomorphisms
$R\rarrow A$ for which the class of $A/R$\+flaprojective $A$\+modules
coincides with the class of projective $A$\+modules.
\end{rems}

 We will denote the class of all $R$\+cotorsion left $A$\+modules
(i.~e., left $A$\+modules that are cotorsion over~$R$) by
$A\Modl^{R\dcot}\subset A\Modl$ and the class of all
$A/R$\+flaprojective left $A$\+modules by
$A\Modl_{A/R\dflpr}\subset A\Modl$.

 Given a left $R$\+module $G$, the left $A$\+module $A\ot_RG$ is
said to be \emph{induced} from~$G$.
 Notice that the $A$\+module $A\ot_RG$ is flat whenever the $R$\+module
$G$ is flat.
 The following lemma describes the Ext spaces from modules induced
from flat modules.

\begin{lem} \label{ext-from-induced-from-flat}
 Let $R\rarrow A$ be a homomorphism of associative rings, $G$ be
a flat left $R$\+module, and $M$ be a left $A$\+module.
 Then there are natural isomorphisms of abelian groups\/
$\Ext^n_A(A\ot_RG,\>M)\simeq\Ext^n_R(G,M)$ for all integers $n\ge1$.
\end{lem}

\begin{proof}
 This is a particular case of~\cite[Proposition~VI.4.1.3]{CE}.
 See also~\cite[formula~(1.1) in Section~1.2]{Pcosh}
or~\cite[Lemma~1.1]{Pdomc}.
\end{proof}

\begin{thm} \label{flaprojective-cotorsion-pair}
 Let $R\rarrow A$ be a homomorphism of associative rings.
 Then \par
\textup{(a)} the pair of classes of left $A$\+modules
$(\sF,\sC)$~$=$ $(A\Modl_{A/R\dflpr}$, $A\Modl^{R\dcot})$
is a complete cotorsion pair in $A\Modl$; \par
\textup{(b)} a left $A$\+module is $A/R$\+flaprojective if and only if
it is a direct summand of an $A$\+module filtered by the left
$A$\+modules $A\ot_RG$ induced from flat left $R$\+modules~$G$.
\end{thm}

\begin{proof}
 This is~\cite[Proposition~3.3]{Pctrl}.
 Notice that the left $A$\+module $A^+=\Hom_\boZ(A,\boQ/\boZ)$
is injective, hence pure-injective, hence cotorsion; so by
Lemma~\ref{cotorsion-restriction-of-scalars}, \,$A^+$ is
a cotorsion left $R$\+module and~\cite[Proposition~3.3]{Pctrl}
is applicable.  {\hbadness=1100\par}

 To spell out some details, one starts with observing that the ``if''
assertion in part~(b) holds by
Lemmas~\ref{ext-from-induced-from-flat} and~\ref{eklof-lemma}.
 Let $\sS_0$ be a set of flat left $R$\+modules such that
$R\Modl_\flat=\Fil(\sS_0)$, as in
Proposition~\ref{flat-modules-deconstructible}.
 Then $R\Modl^\cot=(\sS_0)^{\perp_1}\subset R\Modl$ by
Lemma~\ref{eklof-lemma}.
 Let $\sS\subset A\Modl$ be the set of all left $A$\+modules
$A\ot_RS$, where $S\in\sS_0$.
 Then it follows that $A\Modl^{R\dcot}=\sS^{\perp_1}\subset A\Modl$ by
Lemma~\ref{ext-from-induced-from-flat}.
 Without loss of generality we can assume that ${}_RR\in\sS_0$;
then ${}_AA\in\sS$.
 Now part~(a) is provided by Theorem~\ref{eklof-trlifaj-theorem}(a)
and the ``only if'' implication in part~(b) by
Theorem~\ref{eklof-trlifaj-theorem}(b).
\end{proof}

 We call the complete cotorsion pair
$(\sF,\sC)$~$=$ $(A\Modl_{A/R\dflpr}$, $A\Modl^{R\dcot})$
from Theorem~\ref{flaprojective-cotorsion-pair}
the \emph{flaprojective cotorsion pair} or
the \emph{$A/R$\+flaprojective cotorsion pair} in the category
of left $A$\+modules $A\Modl$.

\begin{cor} \label{relatively-left-perfect-conditions}
 Let $R\rarrow A$ be a homomorphism of associative rings.
 Then the following conditions are equivalent:
\begin{enumerate}
\item a left $A$\+module is cotorsion whenever it is cotorsion
as a module over~$R$;
\item a left $A$\+module is cotorsion if and only if it is cotorsion
as a module over~$R$;
\item all flat left $A$\+modules are $A/R$\+flaprojective;
\item the classes of flat left $A$\+modules and $A/R$\+flaprojective
left $A$\+modules coincide.
\end{enumerate}
\end{cor}

\begin{proof}
 (1)~$\Longleftrightarrow$~(2) holds because all cotorsion $A$\+modules
are always cotorsion over $R$ by
Lemma~\ref{cotorsion-restriction-of-scalars}.

 (3)~$\Longleftrightarrow$~(4) holds because all $A/R$\+flaprojective
$A$\+modules are always flat by
Remark~\ref{flaprojective-remarks}(1).

 (2)~$\Longleftrightarrow$~(4) The pair of classes of left $A$\+modules
($A\Modl_\flat$, $A\Modl^\cot$) is always a (complete) cotorsion pair
in $A\Modl$ by the discussion in the beginning of this section.
 The pair of classes of left $A$\+modules
($A\Modl_{A/R\dflpr}$, $A\Modl^{R\dcot}$) is always a (complete)
cotorsion pair in $A\Modl$ by
Theorem~\ref{flaprojective-cotorsion-pair}(a).

 In a cotorsion pair, any one of the two classes determines the other
class as the respective right/left $\Ext^1_A$\+orthogonal class
in $A\Modl$.
 Therefore, one has $A\Modl_\flat=A\Modl_{A/R\dflpr}$ if and only if
$A\Modl^\cot=A\Modl^{R\dcot}$.
\end{proof}

 One can say that a ring homomorphism $R\rarrow A$ is \emph{relatively
left perfect} if any one of the equivalent conditions of
Corollary~\ref{relatively-left-perfect-conditions} is satisfied.
 Theorem~\ref{A-cotorsion=R-cotorsion-theorem} above and following
corollary claim that the ring homomorphism $R\rarrow A$
is relatively left perfect whenever $A$ is a finitely generated
projective right $R$\+module.

 The following corollary is the second main result of this paper.

\begin{cor} \label{A-flat=A/R-flaprojective-cor}
 Let $R\rarrow A$ be a homomorphism of associative rings such that
$A$ is a finitely generated projective right $R$\+module.
 Then the classes of flat left $A$\+modules and $A/R$\+flaprojective
left $A$\+modules coincide.
 In other words, every flat left $A$\+module is a direct summand of
an $A$\+module filtered by left $A$\+modules $A\ot_RG$ induced from
flat left $R$\+modules~$G$.
\end{cor}

\begin{proof}
 According to Corollary~\ref{relatively-left-perfect-conditions},
the first assertion is an equivalent restatement of
Theorem~\ref{A-cotorsion=R-cotorsion-theorem}.
 The second assertion follows by
Theorem~\ref{flaprojective-cotorsion-pair}(b).
\end{proof}

 Finally, we can formally prove the theorem stated in the introduction.

\begin{proof}[Proof of Theorem~\ref{main-theorem}]
 Part~(a) is Theorem~\ref{A-cotorsion=R-cotorsion-theorem}.
 Part~(b) is Corollary~\ref{A-flat=A/R-flaprojective-cor}.
\end{proof}

\begin{rem}
 Let $R$ be a left perfect ring in the sense of~\cite[Theorem~P]{Bas}.
 Let $R\rarrow A$ be a ring homomorphism making $A$ a finitely
generated projective right $R$\+module.
 Then it follows from Corollary~\ref{A-flat=A/R-flaprojective-cor}
that all flat left $A$\+modules are projective.
 So $A$ is a left perfect ring as well.

 However, this assertion can be proved in a greater generality
with the classical methods.
 The projectivity condition on the right $R$\+module $A$ can be dropped
here; it suffices to assume that $A$ is a finitely generated right
$R$\+module.
 Indeed, by~\cite[Theorem~P(6)]{Bas}, a ring $S$ is left perfect if
and only if all descending chains of cyclic right $S$\+modules
terminate.
 According to~\cite[Theorem~2]{Bj}, this holds if and only if all
descending chains of finitely generated right $S$\+modules terminate.
 It remains to observe that, whenever $A$ is a finitely generated
right $R$\+module, a right $A$\+module is finitely generated if
and only if it is finitely generated as an $R$\+module.
\end{rem}

\begin{quest}
 Of course, neither Theorem~\ref{A-cotorsion=R-cotorsion-theorem}
nor Corollary~\ref{A-flat=A/R-flaprojective-cor} are true \emph{without}
the condition that the right $R$\+module $A$ is finitely generated.
 If this condition is dropped, it suffices to consider the case when
$R=k$ is a field and $A=k[t]$ is the $k$\+algebra of polynomials in
one variable~$t$.
 Then the $A$\+module $k(t)$ of rational functions in~$t$ and
the $A$\+module $k[t,t^{-1}]$ of Laurent polynomials in~$t$ are
flat but not projective.
 So there exist $A$\+modules that are not cotorsion (in fact,
the free $A$\+module $A$ is not cotorsion).
 However, every $k$\+vector space is a cotorsion $k$\+module, so
all $A$\+modules are cotorsion over~$k$.
 Therefore, all $A/k$\+flaprojective $A$\+modules are projective.

 However, we \emph{do not know} whether the assertions of
Theorem~\ref{A-cotorsion=R-cotorsion-theorem} and/or
Corollary~\ref{A-flat=A/R-flaprojective-cor} hold true with
the projectivity assumption on the right $R$\+module $A$ dropped.
 For example, if $A$ is an associative algebra over a commutative ring
$R$ and the $R$\+module $A$ is finitely presented, does it follow that
the ring homomorphism $R\rarrow A$ is relatively (left) perfect?
\end{quest}

 It is clear from Corollary~\ref{relatively-left-perfect-conditions}(1)
or~(2) that the composition of any two relatively left perfect
homomorphisms of associative rings is a relatively left perfect
homomorphism of associative rings again.
 The following example describes another sufficient condition for
a homomorphism of associative rings to be relatively left perfect.

\begin{ex} \label{left-flat-ring-epimorphisms-example}
 This example is due to S.~Bazzoni and is reproduced here with
her kind permission.

 A homomorphism of associative rings $R\rarrow A$ is said to be
a \emph{ring epimorphism} if it is an epimorphism in the category
of associative rings, or equivalently, if the functor of
restriction of scalars $A\Modl\rarrow R\Modl$ is fully faithful.
 A ring homomorphism $R\rarrow A$ is an epimorphism if and only if
the natural map of abelian groups $N\ot_RM\rarrow N\ot_AM$ is
an isomorphism for every right $A$\+module $N$ and
left $A$\+module~$M$.
 A ring epimorphism $R\rarrow A$ is said to be \emph{left flat}
if $A$ is a flat left $R$\+module.
 We suggest the book~\cite[Sections~XI.1\+-3]{Ste} and
the paper~\cite[Section~4]{GL} as the reference sources on
(left flat and arbitrary) ring epimorphisms.

 Let $R\rarrow A$ be a left flat ring epimorphism and $F$ be
a left $A$\+module.
 Then the following conditions are equivalent:
\begin{enumerate}
\item $F$ is flat as an $R$\+module;
\item $F$ is flat as an $A$\+module;
\item $F$ is an $A/R$\+flaprojective $A$\+module.
\end{enumerate}

 Indeed, every flat left $A$\+module is flat as an $R$\+module
since $A$ is a flat left $R$\+module.
 Conversely, for any left $A$\+module $F$ we have isomorphisms
of left $A$\+modules $F\simeq A\ot_AF\simeq A\ot_RF$, since
$R\rarrow A$ is a ring epimorphism.
 Now if $F$ is flat as an $R$\+module, then it is clear that
$F\simeq A\ot_RF$ is flat as an $A$\+module.

 Furthermore, for any homomorphism of associative rings, every
$A/R$\+flaprojective $A$\+module is flat as an $A$\+module by
Remark~\ref{flaprojective-remarks}(1).
 Conversely, by Corollary~\ref{relatively-left-perfect-conditions},
in order to show that all flat left $A$\+modules are
$A/R$\+flaprojective, it suffices to check that all left $A$\+modules
that are cotorsion over $R$ are also cotorsion over~$A$.

 Let $C$ be an $R$\+cotorsion left $A$\+module.
 For any left $A$\+module $C$ we have isomorphisms of left $A$\+modules
 $C\simeq\Hom_A(A,C)\simeq\Hom_R(A,C)$, since $R\rarrow A$ is
a ring epimorphism.
 Now the left $R$\+module $A$ is flat by assumption, while $C$ is
a cotorsion left $R$\+module.
 Applying Lemma~\ref{Hom-module-cotorsion} for the ring $S=A$ and
the $R$\+$A$\+bimodule $K=A$, we conclude that the left $A$\+module
$C\simeq\Hom_R(A,C)$ is cotorsion.

 Alternatively, one can simply say that for, for any flat $A$\+module
$F$, one has $F\simeq A\ot_RF$ and $F$ is a flat $R$\+module; so
the $A$\+module $F$ is $A/R$\+flaprojective by
Theorem~\ref{flaprojective-cotorsion-pair}(b).

 Thus all left flat epimorphisms of associative rings are
relatively left perfect as associative ring homomorphisms.
\end{ex}

\end{document}